\documentclass[a4paper,12pt]{article}
\usepackage[utf8]{inputenc}
\usepackage{amsmath}
\usepackage{amsthm,amssymb,amsfonts}
\usepackage{graphicx}

\usepackage{array}
\usepackage{color}

\newtheorem{theorem}{Theorem}% [section]
\newtheorem{definition}[theorem]{Definition}
\newtheorem{corollary}[theorem]{Corollary}
\newtheorem{proposition}[theorem]{Proposition}
\newtheorem*{remark}{Remark}
\newtheorem{example}[theorem]{Example}
\newtheorem{lemma}[theorem]{Lemma}

\def\supp{{\mathop{\rm supp\, }}}

\title{Wigner transform and quasicrystals}
\author{P. Boggiatto, C. Fern\'andez, A. Galbis, A. Oliaro}

\begin{document}
\maketitle
\begin{abstract}  Quasicrystals, defined as in \cite{olevskii}, are
tempered distributions $\mu$ which satisfy symmetric conditions on
$\mu$ and $\widehat \mu$. This suggests that techniques from
time-frequency analysis could possibly be useful tools in the study
of such structures. In this paper we explore this direction
considering quasicrystals type conditions on time-frequency
representations instead of separately on the distribution and its
Fourier transform. More precisely we prove that a tempered
distribution $\mu$ on ${\mathbb R}^d$ whose Wigner transform,
$W(\mu)$, is supported on a product of two uniformly discrete sets
in ${\mathbb R}^d$ is a quasicrystal. This result is partially
extended to a generalization of the Wigner transform, called
matrix-Wigner transform which is defined in terms of the Wigner
transform and a linear map $T$ on ${\mathbb R}^{2d}$.
\end{abstract}

\section{Introduction}
By a Fourier quasicrystal we mean a tempered distribution
$\mu\in{\mathcal S}^\prime({\mathbb R}^d)$ of the form $\mu =
\sum_{\lambda\in \Lambda}a_\lambda \delta_\lambda$ for which
$\widehat{\mu} = \sum_{s\in S}b_s \delta_s,$ where $\delta_\xi$ is
the mass point at $\xi,$ $\Lambda$ and $S$ are discrete subsets of
${\mathbb R}^d.$ $\Lambda$ and $S$ are called respectively the
support and the spectrum of $\mu.$

The basic examples of Fourier quasicrystals follow from the Poisson
summation formula. N. Lev and A. Olevskii \cite{olevskii} proved
that if a measure $\mu$ on ${\mathbb R}^d$ (which is assumed to be
positive in the case $d > 1$) is a Fourier quasicrystal and both the
support  and the spectrum of $\mu$ are uniformly discrete
 (see Section 2) then there are a lattice $L$ on
${\mathbb R}^d,$ vectors $\theta_j\in {\mathbb R}^d$ and
trigonometric polynomials $P_j$ ($1\leq j\leq N$) such that
$$
\mu = \sum_{j=1}^N P_j(x)\sum_{\lambda\in L+\theta_j}\delta_{\lambda}.$$
\par\medskip
We refer to \cite{kurasov,olevskii2} for examples of quasicrystals
with  other structures. See also \cite{Favorov,olevskii3}.

 The Wigner transform $W(\mu)$ of a tempered distribution $\mu$ gives a description of the time-frequency content of $\mu.$ Hence, it is reasonable to study which information about the structure of $\mu$ follows from the knowledge that $W(\mu)$ is a measure supported on a uniformly discrete set. %We focus on the case $n=1.$
\par\medskip
 The answer to this question does not follow directly from
\cite{olevskii}. Let us assume that $\mu\in{\mathcal
S}^\prime({\mathbb R}^d)$ is an even distribution whose Wigner
transform $W(\mu)$ is a measure on ${\mathbb R}^{2d}$ supported on a
uniformly discrete set. The relation
$W(f)(-\frac{\omega}{2},\frac{x}{2}) = 2^d
\widehat{W(f)}(x,\omega)$ when $f$ is an even square integrable
function (see \cite[4.3,4.2]{grochenig_book}) can be appropriately
extended to even tempered distributions, implying that also the
Fourier transform $\widehat{W(\mu)}$ is a
measure supported on a uniformly discrete set. However, since the Wigner distribution is almost never non negative (see \cite[Theorem 4.4.1]{grochenig_book}) we cannot apply \cite[Theorem 2]{olevskii} to obtain a precise description of $W(\mu).$% See however the proof of Corollary \ref{cor:bounded}.
\par\medskip
Furthermore, from the fact that $W(\mu)$ is supported on a uniformly
discrete set, we cannot even deduce that the support of $\mu$ is
discrete. This is due to the interaction between the Wigner
distribution and the metaplectic operators. More precisely, to every
symplectic matrix $A\in \mbox{Sp}(2,{\mathbb R})$ one can associate
a unitary operator $T_A$ acting on $L^2({\mathbb R})$ (denoted
$\mu(A)$ in \cite{folland}) such that
    \begin{equation}\label{eq:wigner-metaplectic}
    W\big(T_A f, T_A g\big)(z) = W(f,g)(A^\ast z)\ \ \forall z = (x,\omega)\in {\mathbb R}^2.\end{equation}
Moreover, $T_A$ extends to an isomorphism on ${\mathcal
S}^\prime({\mathbb R})$ and (\ref{eq:wigner-metaplectic}) holds for
$f,g\in {\mathcal S}^\prime({\mathbb R}).$ We refer to
\cite{folland},  Propositions 4.27 and 4.28. Let us now
consider the Dirac comb $\mu=\sum_{n\in {\mathbb Z}}\delta_n.$ When
$$A = \left(\begin{array}{rr}\cos\theta & -\sin\theta\\ \sin\theta &
\cos\theta\end{array}\right), \, \theta \in (-\pi, \pi)$$ $T_A \mu$
is the fractional Fourier transform ${\mathcal F}^\alpha \mu,$ where
$\alpha=\frac{2}{\pi}\theta \in (-2,2).$ In this case $W(T_A\ \mu)$
is a rotation of $W(\mu),$ hence it is supported on a uniformly
discrete set of ${\mathbb R}^2.$ However, $\theta$ can be chosen
such that $\mbox{supp}\ T_A\ \mu ={\mathbb R}$ (see \cite[Theorem
1.2]{viola}).

\par\medskip
Our objective will be to obtain information about $\mu\in{\mathcal
S}^\prime({\mathbb R}^d)$ from the fact that the Wigner transform
$W(\mu)$ is a measure supported on the product of two uniformly
discrete subsets of ${\mathbb R}^d.$  The main result of the paper is as
follows.
\begin{theorem}\label{th:main}
 Let $\mu\in {\mathcal S}^\prime({\mathbb R}^d)$ satisfy
$ W(\mu) = \sum_{(r, s)\in A\times B} c_{r,s} \delta_{(r,s)}$ where $A, B$ are uniformly discrete sets. Then $\mu$ and $\widehat{\mu }$ are measures with supports contained in $A$ and $B$ respectively.
\end{theorem}
The proof is contained in section \ref{sec:proof}. To facilitate the
reading, we have preferred to include the complete proof in
dimension $d=1$ and then indicate the necessary modifications to
obtain the result in arbitrary dimension.

 The usefulness of time-frequency representations in the
study of quasicrystals is not limited to the Wigner transform,
actually in section \ref{sec:matrix} we consider a generalization
of the Wigner transform, called matrix-Wigner transform which is
defined in terms of the Wigner transform and a linear map $T$ on
${\mathbb R}^{2d}$  and contains for particular choices
of $T$ most of the classic time-frequency representations. First we
obtain some results that relate the support of two distributions
$\mu, \nu$ (or $\widehat{\mu}, \widehat{\nu}$) with that of the
cross matrix-Wigner transform $W_T(\mu,\nu)$, thus generalizing the
information contained in Theorem \ref{th:main} relative to the
supports. Then, we focus on the one-dimensional case and obtain a
version of Theorem \ref{th:main} for the cross matrix-Wigner
transform.
%The generality of the choice of the linear map $T$ shows that our
%results apply to most of the classical time-frequency representations,
We remark moreover that a particular choice of $T$ connects our
framework to that of Lev and Olevskii \cite{olevskii}, which
essentially corresponds to the case of the Rihaczek representation.

We suppose that the link between quasicrystals and
time-frequency analysis presented in this paper could lead to
further developments, both in view of recent results e.g. as in
\cite{olevskii4}, as well as in the direction of using specific
time-frequency representations to enlighten particular features of
quasicrystals structures.

\section{Notation}
\par\medskip\noindent
We use brackets $\langle\mu,g\rangle$ to denote the extension to
${\mathcal S}'({\mathbb R}^d)\times {\mathcal S}({\mathbb R}^d)$ of the inner
product $\begin{displaystyle}\langle f,g\rangle = \int_{{\mathbb R}^d} f(t)
    \overline{g(t)} dt \end{displaystyle}$ on $L^2({\mathbb R}^d).$ We will write $\langle g, \mu\rangle$ instead of $\overline{\langle\mu, g\rangle}.$
\par\medskip
The cross-Wigner distribution of $f, \, g \in L^2({\mathbb R}^d)$ is
$$W(f,g)(x, \omega)=\displaystyle \int_{{\mathbb R}^d}f(x+\frac{t}{2})\overline{g(x-\frac{t}{2})} e^{-2\pi i  \omega t}dt,\ \ x,\omega\in {\mathbb R}^d.$$ It happens that $W(f,g) \in
L^2({\mathbb R}^{2d}).$ Moreover, the cross-Wigner distribution maps ${\mathcal S}({\mathbb R}^d)\times {\mathcal S}({\mathbb R}^d)$ into ${\mathcal S}({\mathbb R}^{2d}).$ The Wigner distribution of $f\in L^2({\mathbb R}^d)$ is $W(f):=W(f,f).$ It is a quadratic representation of the signal $f$ both in time and frequency and it is covariant, which means that
$$
W(T_\alpha M_\beta f)(x,\omega) = W(f)(x-\alpha, \omega-\beta).$$ Here, $T_\alpha$ and $M_\beta$ are the translation and modulation operators, defined by
$$M_{\omega}f(t)= e^{2 \pi \imath \omega t}f(t) \, \, \, {\rm and } \ \
\, T_xf(t)=f(t-x).$$ The cross Wigner distribution can be extended as a continuous map from ${\mathcal S}'({\mathbb R}^d)\times {\mathcal S}'({\mathbb R}^d)$ into ${\mathcal S}'({\mathbb R}^{2d})$  as follows \cite[4.3.3]{grochenig_book}
$$
\langle W(\mu,\nu), \phi\rangle = \langle \mu\otimes \overline{\nu},
{\mathcal T}_s^{-1}{\mathcal F}_2^{-1}\phi\rangle$$ for any $\phi\in
{\mathcal S}({\mathbb R}^{2d}),$ where ${\mathcal F}_2$ denotes the
partial Fourier transform
$$
{\mathcal F}_2F(x, \omega)=\int_{{\mathbb R}^d}F(x,t) e^{-2\pi i
\omega   t}\ dt,\ \ x, \omega \in {\mathbb R}^d
$$
and ${\mathcal T}_s$ is the symmetric coordinate change defined by
\begin{equation}\label{Ts}
{\mathcal T}_sF(x,t)=F(x+\frac{t}{2},x-\frac{t}{2}), \, \, x,t\in
{\mathbb R}^d. \end{equation}
This extension satisfies Moyal's
formula, that is, for any functions $\phi,\psi\in {\mathcal
S}({\mathbb R}^d)$ one has (see for instance
\cite[4.3.2]{grochenig_book})
$$
\langle W(\mu,\nu), W(\phi,\psi)\rangle = \langle\mu, \phi\rangle\cdot \langle\psi,\nu\rangle.$$
\par\medskip\noindent
A set $A\subset {\mathbb R}^d$ is said to be {\it uniformly
discrete} (u.d. from now on) if there is $\delta > 0$ such that
$|r-s|\geq \delta$ whenever $s,r\in A, s\neq r.$

\section{The proof of Theorem \ref{th:main}}\label{sec:proof}
The next result is well-known and will play a role in the proof of Theorem \ref{th:main}. We include a proof for the convenience of the reader. As usual, for a multiindex $\alpha\in {\mathbb N}_0^d$ we denote its length by $|\alpha| = \alpha_1+\ldots+\alpha_d.$ $B_\varepsilon$ stands for the ball with radius $\varepsilon$ centered at the origin.
\begin{lemma}\label{lem:sum_deltas}
Let $\mu\in {\mathcal S}^\prime({\mathbb R}^d)$ be a tempered distribution with u.d. support $A.$ Then there are $N\in {\mathbb N}_0$ and complex numbers $\left\{b^\alpha_r:\ 0\leq |\alpha|\leq N, r\in A\right\}$ such that
$$
\mu = \sum_{r\in A}\sum_{|\alpha|\leq N} b^\alpha_r \delta^{(\alpha)}_r$$ on ${\mathcal S}({\mathbb R}^d).$ Moreover
$$
\sup_{|\alpha|\leq N}\sup_{r\in A}\left|b^\alpha_r\right|\big(1+|r|\big)^{-N} < \infty.$$
\end{lemma}
\begin{proof}
Let $N\in {\mathbb N}_0$ and $C > 0$ satisfy
$$
|\mu(f)|\leq C\sup_{x\in {\mathbb R}^d}\sup_{|\alpha|\leq N}\left(1+|x|\right)^N\left|f^{(\alpha)}(x)\right|$$ for every $f\in {\mathcal S}({\mathbb R}^d).$ Since $\mu$ has u.d. support we find complex numbers $\left\{b^\alpha_r:\ |\alpha|\leq N, r\in A\right\}$ such that
\begin{equation}\label{eq:discrete_support}
\mu(\varphi) = \sum_{r\in A}\sum_{|\alpha|\leq N} b^\alpha_r \delta^{(\alpha)}_r(\varphi)\end{equation} for every $\varphi\in {\mathcal D}({\mathbb R}^d).$ Let us check that the right  hand side on (\ref{eq:discrete_support}) defines a tempered distribution. We take $0 < \varepsilon < \inf\{|r-r'|:\ r,r'\in A\}.$ For $|\alpha|\leq N$ we take $\varphi_\alpha\in {\mathcal D}(B_\varepsilon)$ such that $\varphi^{(\alpha)}_\alpha(0) = 1$ and $\varphi^{(\beta)}_\alpha(0) = 0$ for $|\beta|\leq N, \beta\neq \alpha.$ Then
$$
\left|b^\alpha_r\right| = \left|\mu\big(T_r\varphi_\alpha\big)\right| \leq \tilde{C}\big(1+|r|\big)^{N}$$ where the constant $\tilde{C}$ does not depend on $\alpha$ or $r\in A.$ Therefore, the right hand side in (\ref{eq:discrete_support}) defines a tempered distribution. Finally, the density of ${\mathcal D}({\mathbb R}^d)$ on ${\mathcal S}({\mathbb R}^d)$ gives the conclusion.
\end{proof}

\begin{lemma}\label{lem:support}
    Let $\mu\in {\mathcal S}'({\mathbb R}^d)$ satisfy $ W(\mu) = \sum_{(r,
s)\in A\times B} c_{r,s} \delta_{(r,s)}$ where $A, B$ are u.d. sets. Then $\mbox{supp}\ \mu\subset
A$ and $\mbox{supp}\ \widehat{\mu}\subset B.$ Moreover,
$\frac{r_1+r_2}{2}\in A$ for any $r_1, r_2\in \mbox{supp}\ \mu.$
\end{lemma}
\begin{proof}
The inclusions $\mbox{supp}\ \mu\subset A$ and $\mbox{supp}\ \widehat{\mu}\subset B$ follow from standard properties of the Wigner transform. Let us now fix $r_1, r_2\in \mbox{supp}\ \mu$ and consider $\nu:=T_{-r_1}\mu,$ so that $0, r_0:=r_2-r_1\in \mbox{supp}\ \nu.$ From the covariance property of the Wigner transform we obtain a representation
$$
 W(\nu) = \sum_{(r, s)\in A_1\times B} b_{r,s} \delta_{(r,s)}$$ where $A_1 = A-r_1.$ Since $\nu$ is a tempered distribution with u.d. support contained in $A_1$ we have
 $$
 \nu = \sum_{r\in A_1}\sum_{|\alpha|\leq N} b_r^\alpha \delta_r^{(\alpha)}$$ for some $N\in {\mathbb N}$ and $b_r^\alpha\in {\mathbb C}.$ We aim to prove that $\frac{r_0}{2}\in A_1,$ which means $\frac{r_1+r_2}{2}\in A$ as desired. Proceeding by contradiction, let us assume that $\frac{r_0}{2}\notin A_1.$ We choose $0 < \varepsilon < \mbox{dist}\left(\frac{r_0}{2},A_1\right).$ Since $0\in \mbox{supp}\ \nu$ we can find $g\in {\mathcal D}(B_\varepsilon)$ real-valued and satisfying $\langle \nu, g\rangle \neq 0.$ For every $\alpha\in {\mathbb N}_0^N$ with $|\alpha|\leq N$ let $f_\alpha\in {\mathcal D}(r_0+B_\varepsilon)$ be a real-valued function such that $f_\alpha^{(\alpha)}(r_0) = (-1)^{|\alpha|}$ but $f_\alpha^{(\beta)}(r_0) = 0$ for any $\beta\neq \alpha.$ Now we observe that
 $$
 \begin{array}{*2{>{\displaystyle}l}}
    \langle g, \nu\rangle \ b_{r_0}^\alpha & = \langle g, \nu\rangle \cdot \langle\nu, f_\alpha\rangle = \langle W(\nu), W(f_\alpha,g)\rangle \\ & \\ & = \sum_{(r, s)\in A_1\times B} b_{r,s}\int_{{\mathbb R}^d}e^{2\pi its}f_\alpha(r+\frac{t}{2})g(r-\frac{t}{2})\ dt.

 \end{array}
 $$ Now,
 $$
 f_\alpha(r+\frac{t}{2})g(r-\frac{t}{2})\neq 0$$ implies that $$ r+\frac{t}{2}\in r_0+B_\varepsilon\mbox{ and } r-\frac{t}{2}\in B_\varepsilon,$$ and hence $$r \in (\frac{r_0}{2}+B_\varepsilon) \cap A = \emptyset.$$ We conclude $b_{r_0}^\alpha = 0$ for every $|\alpha|\leq N,$ which is a contradiction since $r_0\in \mbox{supp}\ \nu.$
 \end{proof}
\par\medskip
\begin{remark}{\rm
Under the hypothesis of Lemma \ref{lem:support} the set $$\frac{\mbox{supp}\ \mu+\mbox{supp}\ \mu}{2}$$ is u.d.. We note however that there are u.d. sets $A$ such that $\frac{A+A}{2}$ has accumulation points. As an example in dimension $d=1$ we can consider $A = \{n+\frac{1}{|n|}:\ n\in {\mathbb Z}\setminus\{0\}\},$ for which $0$ is an accumulation point of $\frac{A+A}{2}.$
} $\Box$
\end{remark}
\par\medskip
The next elementary result will be used in the proof of Theorem \ref{th:main}. We omit the proof.

\begin{lemma}\label{lem:double-series}
Let $A\subset {\mathbb R}^d$ be a u.d. set. Then for every $\alpha > 0$ there exists $\beta > 0$ such that$$
\sum_{r\in A}\sum_{s\in A}\left(1+|r|\right)^{\alpha}\left(1+|s|\right)^{\alpha}\left(1+|r+s|\right)^{-\beta}\left(1+|r-s|\right)^{-\beta} < \infty.$$
\end{lemma}
\par\vskip 1cm
{\it Proof of Theorem \ref{th:main} in dimension $d=1$:}

Let  us write $D$ for the support of $\mu,$ which is contained in
$A.$ According to Lemma \ref{lem:sum_deltas} we can put
$$\mu=\sum_{r\in D}\sum_{j=0}^N a_r^j \delta^{(j)}_r,$$ with $a_r^j\in {\mathbb C}.$ We now assume $N\geq 1$ and show that $a_r^N = 0$ for all $r\in D.$
\par\medskip
For any real-valued functions $\phi_1, \phi_2 \in {\mathcal
S}({\mathbb R})$ we have, for $\phi = \phi_1\otimes\phi_2,$
$$
\begin{array}{*2{>{\displaystyle}l}}
    \left({\mathcal T}_s^{-1}{\mathcal F}_2^{-1}\phi\right)(u,v) & = \left({\mathcal F}_2^{-1}\phi\right)\left(\frac{u+v}{2},u-v\right) \\ & \\ & = \phi_1\left(\frac{u+v}{2}\right)\widehat{\phi_2}(v-u),
    \end{array}$$
    hence
$$
\langle W(\mu), \phi_1\otimes\phi_2\rangle = \langle \mu_u, \langle
\mu_v,
\phi_1\left(\frac{u+v}{2}\right)\overline{\widehat{\phi_2}\left(v-u\right)}\rangle\rangle.$$

A simple calculation gives
\begin{equation}\label{eq:wigner}
\begin{array}{*1{>{\displaystyle}c}} \langle W(\mu), \phi_1\otimes \phi_2 \rangle =\\ \\ \sum_{j, k =0}^N
\sum_{\ell=0}^j\sum_{m=0}^k (-1)^{j+k}\lambda_{j,k}^{\ell,m}\sum_{r,
\, s \in D} a^k_{s}\overline a^j_r\phi_1^{(\ell+m)}\left(\frac{r+s}{2}\right)\overline{\widehat{\phi_2}}^{(j+k-\ell-m)}(r-s),
\end{array}\end{equation}
where
$$\lambda_{j,k}^{\ell,m} = {j \choose \ell} {k \choose m} (-1)^{k-m} \frac{1}{2^{m+\ell}}.$$
\par\medskip
Since
\begin{equation}\label{eq:bound-coeff}
\sup_{r\in D}\left|a^j_r\right|\left(1+|r|\right)^{-N} < \infty
\end{equation}
for every $0\leq j\leq N,$ an application of Lemma
\ref{lem:double-series} permits to conclude that the double series
in the right hand side of (\ref{eq:wigner}) is absolutely
convergent. We take $$0 < \varepsilon < \delta(A):= \inf\{|r -
r^\prime|:\ r, r^\prime\in A,\ r\neq r^\prime\}.$$ By Lemma
\ref{lem:support}, $\frac{r+s}{2}\in A$ for any $r, s \in D,$
therefore $\{\frac{r+s}{2}: r, \, s \in D \}$ has no accumulation
points. Fix $r_0 \in D$ and choose $\phi_1\in {\mathcal
D}(r_0-\varepsilon, r_0+\varepsilon)$ such that
$\phi_1^{(n)}(r_0)=0$ for $n=0, \dots , 2N-1$ whereas
$\phi^{(2N)}_1(r_0)\neq 0.$ Then, for $\phi_2 \in {\mathcal
S}({\mathbb R})$ (real valued), we have

$$\langle W(\mu), \phi_1
\otimes \phi_2 \rangle = \frac{1}{2^{2N}}\phi_1^{(2N)}(r_0)\sum_{r,
\, s \in D(r_0)}\overline{a^N_r} a^N_s \overline{\widehat{\phi_2}}(r-s).$$

Here $D(r_0):=\{(r, s):\ r,s\in D;\ r+s=2r_0\}.$ Hence, if $\phi_2$ has also compact support,
$$
\begin{array}{*2{>{\displaystyle}l}}
    \left| \phi_1^{(2N)}(r_0)\sum_{(r, s) \in D(r_0)}\overline{a^N_r}a^N_s \overline{\widehat{\phi_2}}(r-s)
\right | & = 2^{2N}\left| \langle W(\mu), \phi_1\otimes \phi_2
\rangle \right|\\ & \\ & \leq C ||\phi_1||_\infty
||\phi_2||_\infty,\end{array}$$ where the constant $C$ depends on
the (compact) support of $\phi_1\otimes \phi_2.$

Next, we fix  $\psi \in {\mathcal D}(-\varepsilon, \varepsilon)$
such that $\psi^{(n)}(0)=0$ for $n=0, \dots , 2N-1$ and
$\psi^{(2N)}(0)=1,$ and for each $t\geq 1,$ let us consider $\psi^t(
x):=\psi(t x)$ and $\phi^t_1(x)=\psi^t(x-r_0).$
Hence, as the
supports of  the $\phi^t_1$'s shrink as $t$ increases, we have, for
every $t \geq 1,$
$$
t^{2N}\left| \sum_{(r, s) \in D(r_0)}\overline{a^N_r}
a^N_s \overline{\widehat{\phi_2}}(r-s)\right|\leq C
||\psi||_\infty ||\phi_2||_\infty,$$ where $C$ depends on the
support of $\phi_1^1\otimes \phi_2.$ Taking limits as $t$ goes to
infinity we conclude that
\begin{equation}\label{eq:sum-0}
\sum_{(r, s) \in D(r_0)}\overline{a^N_r} a^N_s
\overline{\widehat{\phi}_2}(r-s) = \sum_{r
\in D_0}\overline{a^N_r} a^N_{2r_0-r}
\overline{\widehat{\phi}_2}(2(r-r_0)) = 0
\end{equation}
for every $\phi_2\in {\mathcal D}({\mathbb R}),$ where
$$
D_0=\{r\in D : \text{there\ exists\ }s\in D\text{\ with\ }r+s=2r_0\}.$$ From (\ref{eq:bound-coeff}) it follows that the map
$$
\phi \mapsto \sum_{(r, s) \in D(r_0)}\overline{a^N_r}
a^N_s \overline{\widehat{\phi}}(r-s)$$ defines a
tempered distribution, which coincides with $$ \sum_{(r, s) \in
    D(r_0)}\overline{a^N_r} a^N_s e^{-2\pi
    i(r-s)x}={\mathcal F}\left(\sum_{(r, s) \in
    D(r_0)}\overline{a^N_r} a^N_s \delta_{r-s}\right)$$ the
series being convergent in ${\mathcal S}'({\mathbb R}).$ Hence, by density, equation (\ref{eq:sum-0})
holds for every $\phi_2\in {\mathcal S}({\mathbb R}).$ Now we consider $\phi_2\in {\mathcal S}({\mathbb R})$ such that $\rm supp\
\widehat \phi_2$ is a so small compact set that it does not contain
other points of $(D-r_0)$ other than possibly $0$, a fact which is
possible since $D \subset A$ has no accumulation points. Then equation (\ref{eq:sum-0}) reduces to $\overline{a^N_{r_0}}a^N_{r_0}=0$, i.e. $a^N_{r_0}=0$.
\par\medskip
Proceeding by recurrence, we finally get that $a^j_r=0$ for all $r\in D$ whenever $j\geq 1.$ This proves that $\mu$ is a measure, as desired.
\par\medskip The conclusion for $\widehat{\mu}$ now follows from $W(\widehat{\mu})(x,\omega) = W(\mu)(-\omega,x)$ (see \cite[4.3.2]{grochenig_book}). $\Box$
\par\medskip
Let $A, B$ u.d. subsets of ${\mathbb R}.$ It is well-known that $$T:= \sum_{(r,s)\in A\times B}c_{r,s}\delta_{(r,s)}$$ defines a tempered distribution if, and only if, for some $N\geq 0$ $$
\sup_{(r,s)\in A\times B}|c_{r,s}|\left(1+|r|+|s|\right)^{-N} < +\infty.$$

If $T$ is obtained as the Wigner transform of a tempered distribution then a more restrictive condition on the coefficients is satisfied.

\begin{corollary}\label{cor:bounded} Let $\mu\in {\mathcal S}^\prime({\mathbb R})$ satisfy
    $ W(\mu) = \sum_{(r, s)\in A\times B} c_{r,s} \delta_{(r,s)}$ where $A, B$ are u.d. sets. Then
    $$
    \sup_{(r,s)\in A\times B}|c_{r,s}| < +\infty.$$
\end{corollary}
\begin{proof}
According to Theorem \ref{th:main} we can apply \cite[Theorems 1,3]{olevskii} to conclude that there are $a > 0$ and $N\in {\mathbb N}$ such that $\mu = \sum_{j=1}^N \mu_j,$ where each $\mu_j$ is a finite linear combination of time-frequency shifts of $\sum_{n\in {\mathbb Z}}\delta_{na}.$ The conclusion follows from the properties of the cross-Wigner distribution (see for instance \cite[4.3.2(c)]{grochenig_book} in the $L^2$ setting) and the fact that, for $\Lambda = a{\mathbb Z},$
$$
W(\sum_{\lambda\in \Lambda}\delta_{\lambda}) = \frac{1}{2a}\sum_{\lambda\in \Lambda}\delta_{\lambda}\otimes \sum_{\lambda^\ast\in \Lambda^\ast}\delta_{\frac{\lambda^\ast}{2}} + \frac{1}{2a}\sum_{\lambda\in \Lambda}\delta_{\lambda+\frac{a}{2}}\otimes \sum_{\lambda^\ast\in \Lambda^\ast}e^{-\pi i \lambda^\ast}\delta_{\frac{\lambda^\ast}{2}},$$ where $\Lambda^\ast = \frac{1}{a}{\mathbb Z}$ is the dual lattice.
\end{proof}
 Our next goal is to adapt the previous arguments to
the case of arbitrary dimension $d$. We first need a technical
lemma. For any $\gamma\in {\mathbb N}_0^d$ we denote
$$
F_{\gamma}^d = \left\{(\alpha, \beta)\in {\mathbb N}_0^d\times {\mathbb N}_0^d:\ |\alpha| = |\beta| = |\gamma|,\ \alpha + \beta = 2\gamma\right\}.$$

\begin{lemma}\label{lem:several}
    Let $N\geq 1$ and let us assume that the family
     of complex numbers $\left\{a^\gamma\right\}$ indexed by  $\gamma\in {\mathbb N}_0^d$ satisfies
    \begin{equation}\label{eq:condition}
        \sum_{(\alpha, \beta)\in F_{\gamma}^d} a^\alpha \overline{a^\beta} = 0\end{equation} for every $\gamma\in {\mathbb N}_0^d$ with $|\gamma|=N.$ Then $a^\gamma = 0$ whenever $|\gamma|=N.$
\end{lemma}
\begin{proof}
    We proceed by induction on $d.$ For $d=1$ the lemma is obvious since
condition (\ref{eq:condition}) means $|a^\gamma|^2=0$ for
$\gamma=N.$ Let us now assume that the lemma holds in dimension
$d-1$ ($d\geq 2$) for every $N\geq 1$ and let $\left\{a^\gamma:\
\gamma\in {\mathbb N}_0^d\right\}\subset {\mathbb C}$ be given such
that condition (\ref{eq:condition}) is satisfied. Now we proceed by
induction on $k$ to prove that $a^\gamma = 0$ whenever $|\gamma|=N$
and at least one component of $\gamma$ equals $k.$ Let
us first assume that $k = 0,$ so $\gamma$ has at most $d-1$ non-null
components. Without loss of generality, we can assume
$\gamma_d = 0.$ For every $\delta'\in {\mathbb N}_0^{d-1}$ we define
    $$
    b^{\delta'} = a^{(\delta', 0)}.$$ For any $\delta'\in {\mathbb N}_0^{d-1}$
with $|\delta'| = N$ we put $\delta = (\delta', 0).$ Then
    $$
    \sum_{(\alpha', \beta')\in F_{\delta'}^{d-1}} b^{\alpha'}
\overline{b^{\beta'}} = \sum_{(\alpha, \beta)\in F_{\delta}^d}
a^\alpha \overline{a^\beta} = 0.$$ Our hypothesis on the validity of
the lemma in dimension $d-1$ permits to conclude that $a^\gamma =
0.$ Let us now assume that $a^\gamma = 0$ whenever $|\gamma|=N$ and
at least one component of $\gamma$ is less than $k$
($k\geq 1$) and let us fix $\gamma\in {\mathbb N}_0^d$ such that
$|\gamma|=N$ and at least one component equals $k.$ We
can assume $\gamma_d = k$ and $\gamma_j \geq k$ for every $1\leq
j\leq d-1$ (otherwise $a^\gamma = 0$). This implies $1 \leq k < N.$
For every $\delta'\in {\mathbb N}_0^{d-1}$ we define
    $$
    b^{\delta'} = a^{(\delta', k)}$$ and we put $\delta = (\delta', k).$ Then,
    for every $\delta'\in {\mathbb N}_0^{d-1}$ with $|\delta'| = N-k,$
    $$
    \sum_{(\alpha', \beta')\in F_{\delta'}^{d-1}} b^{\alpha'} \overline{b^{\beta'}} = \sum_{(\alpha, \beta)\in F_{\delta}^d} a^\alpha \overline{a^\beta} = 0.$$
    We observe that $|\delta| = N$ and condition $(\alpha, \beta)\in
F_{\delta}^d$ implies $\alpha_d + \beta_d = 2k,$ hence we can assume
$\alpha_d = \beta_d = k.$ Otherwise some of the coefficients
$\alpha_d, \beta_d$ is less than $k$, from where it follows
$a^\alpha \overline{a^\beta}= 0.$ Our hypothesis on the validity of
the lemma in dimension $d-1$ (applied to $N-k$ instead of $N$)
permits to conclude that $b^{\delta'}=0$ for any $\delta'\in
{\mathbb N}_0^{d-1}$ such that $|\delta'|=N-k.$ In particular,
$a^\gamma = b^{\gamma'} = 0.$ Here $\gamma' = (\gamma_1, \ldots,
\gamma_{d-1}).$  The proof is finished.
\end{proof}

\par\vskip 1cm
{\it Proof of Theorem \ref{th:main} for arbitrary $d$:}

    As in the case $d=1$ we only need to check the statement concerning
$\mu.$ Let us write $S_\mu$ for the support of $\mu,$ which is
contained in $A.$ According to Lemma \ref{lem:sum_deltas} we put
    $$\mu=\sum_{r\in S_\mu}\sum_{|\alpha|\leq N} a_r^\alpha \delta^{(\alpha)}_r,$$ with $a_r^\alpha\in {\mathbb C}.$ Our aim is to show that $a_r^\alpha = 0$ for all $r\in S_\mu$ and $|\alpha|\geq 1.$
    \par\medskip
Recall that $\{\frac{r+s}{2}: r, \, s \in S_\mu \}$ has no
accumulation points. Fix $r_0 \in S_\mu$ and $\gamma\in {\mathbb
N}_0^d$ with $|\gamma|= N$ and choose a smooth function $\phi_1$
supported on a sufficiently  small neighborhood of $r_0$ and
satisfying $\phi_1^{(2\gamma)}(r_0)\neq 0$ while
$\phi_1^{(\alpha)}(r_0)=0$ for each $\alpha\neq 2\gamma.$ Then, for
every   $\phi_2 \in {\mathcal S}({\mathbb R}^d)$ (real valued), we
have
    $$\langle W(\mu), \phi_1 \otimes \phi_2 \rangle =
\phi_1^{(2\gamma)}(r_0)\left(\frac{1}{2}\right)^{2|\gamma|}\sum_{(r,s)
\in D(r_0)}\Big(\sum_{(\alpha,\beta)\in F^d_\gamma}a^\alpha_r
\overline{a^\beta_s}\Big)\cdot
\overline{\widehat{\phi_2}}(r-s),$$ where
$D(r_0):=\{(r, s):\ r,s\in S_\mu,\ r+s=2r_0\}.$ Proceeding as in the
case $d=1$ we obtain
    $$
    \sum_{(\alpha,\beta)\in F^d_\gamma}a^\alpha_{r_0}
\overline{a^\beta_{r_0}} = 0.$$ An application of Lemma
\ref{lem:several} gives $a_{r}^\gamma = 0$ whenever $|\gamma| = N$
and $r\in S_\mu.$ Now a recurrence argument shows that $a_{r}^\gamma
= 0$ for every $r\in S_\mu$ and $\gamma\in {\mathbb N}_0^d$ with
$|\gamma|\geq 1.$ The conclusion follows. $\Box$

\section{The matrix-Wigner transform}\label{sec:matrix}

A natural generalization of the hypothesis of Theorem \ref{th:main}
is the case where different input functions or distributions $\mu$
and $\nu$ are considered for the Wigner transform. This situation is
more involved but still some results can be obtained. Furthermore we
shall consider a generalization of the Wigner transform, called {\it
matrix-Wigner transform}, see \cite{Bayer}, which, using a
composition with linear maps, will yield a unifying framework
connecting our results to those of \cite{olevskii},
\cite{olevskii2}. We need some preliminaries.

We begin by recalling the following notations. For a set $E\in
\mathbb R^{2d}=\mathbb R^d_x\times\mathbb R^d_\omega$ we indicate
the projections on the $x$ and $\omega$-coordinates as:
$$
\begin{array}{c}
\Pi_1(E)=\{x\in\mathbb R^d_x: \exists \, \omega\in\mathbb R^d_\omega
\hbox{\
such\ that \ } (x,\omega)\in E\},\\
\Pi_2(E)=\{\omega\in\mathbb R^d_\omega: \exists \, x\in\mathbb R^d_x
\hbox{\ such\ that \ } (x,\omega)\in E\}.
\end{array}
$$

When it is clear from the context we shall however omit the
subscripts $x$ and $\omega$ in $\mathbb R^d_x$ and $\mathbb
R^d_\omega$. We recall without proof the following well-known
property which will be used later.

\begin{proposition}
Setting ${\mathcal F}_1F(\nu, t)=\int_{{\mathbb R}^d}F(x,t) e^{-2\pi i
\nu x}\ dx$ and ${\mathcal F}_2F(x,\omega)=\int_{{\mathbb R}^d}F(x,t)
e^{-2\pi i \omega t}\ dt$ for $F\in \mathcal S(\mathbb R^{2d})$,
with usual extensions to $\mathcal S'(\mathbb R^{2d})$, the partial
Fourier transforms $\mathcal F_1, \mathcal F_2$ are bicontinuous
isomorphisms from $\mathcal S(\mathbb R^{2d})$ to $\mathcal
S(\mathbb R^{2d})$ and from $\mathcal S'(\mathbb R^{2d})$ to
$\mathcal S'(\mathbb R^{2d})$.
\end{proposition}

We need now to discuss some properties concerning supports.

\begin{lemma}\label{Proj}
Suppose that $\Psi\in\mathcal S'(\mathbb R^{2d})$. If
$\Pi_1\supp\Psi$ or $\Pi_1\supp\mathcal F_2 \Psi$ are u.d. sets in
$\mathbb R^d$, then $\Pi_1\supp \Psi=\Pi_1\supp\mathcal F_2 \Psi$
(and therefore both are u.d.).
\end{lemma}

\begin{proof}
Let $I$ be an interval in $\mathbb R^d$ (i.e. the cartesian product
of $d$ open intervals of $\mathbb R$). Preliminarly we observe that
a distribution $\Psi\in\mathcal S'(\mathbb R^{2d})$ vanishes on the
strip $I\times \mathbb R^d$ if and only if the restriction of
$\mathcal F_2\Psi$ to the same strip also vanishes. Indeed if the
distribution $\Psi$ vanishes on $I\times \mathbb R^d$, then for
every $\phi\in\mathcal S(\mathbb R^{2d})$ with $\supp\phi\subset
I\times\mathbb R^d$, we have
$$
0=\langle \Psi, \phi \rangle = \langle \mathcal F_2\Psi, \mathcal
F_2\phi\rangle,
$$
which means that also $\mathcal F_2 \Psi$ vanishes on $I\times
\mathbb R^d$ because $\mathcal F_2:\mathcal S(\mathbb
R^{2d})\longrightarrow\mathcal S(\mathbb R^{2d})$ is a bijection
which preserves the inclusion of the supports in $I\times\mathbb
R^d$. The converse is analogous.

Let us suppose now that $\Pi_1\supp \Psi$ is a u.d. set and $x_0
\notin \Pi_1\supp \Psi$, then there exists an open interval
$I\subset \mathbb R^d$ containing $x_0$ such that $\Psi$ vanishes on
$I\times \mathbb R^d$. From the first part of this proof we know
that also $\mathcal F_2\Psi$ vanishes on $I\times \mathbb R^d$ and
therefore $x_0\notin\Pi_1\supp\mathcal F_2\Psi$. This proves the
inclusion
$$
\Pi_1\supp\mathcal F_2 \Psi \subseteq \Pi_1\supp\Psi,
$$
which in particular shows that also $\Pi_1\supp\mathcal F_2 \Psi$ is
u.d.

The opposite inclusion $\Pi_1\supp\Psi \subseteq \Pi_1\supp\mathcal
F_2 \Psi$, is proved by the same argument and we have therefore
$$
\Pi_1\supp\Psi = \Pi_1\supp\mathcal F_2 \Psi.
$$

Finally the case where $\Pi_1\supp \mathcal F_2 \Psi$ is a u.d. set
can be proved in analogous way.
\end{proof}

\begin{remark}{\rm
We observe that the hypothesis of u.d.ness of either
$\Pi_1\supp\Psi$ or $\Pi_1\supp\mathcal F_2 \Psi$ in the previous
proposition is essential. Consider for example
$$
\Psi=\sum_{n\in\mathbb N} \delta_{1/n}(x)\otimes\delta_{n}(\omega).
$$
Then $0$ belongs to $\Pi_1\supp \mathcal F_2\Psi$ but not to
$\Pi_1\supp \Psi$.}
\end{remark}

In order to treat the matrix-Wigner transform, to be defined later,
we need to consider some properties of linear maps in connection
with supports.

\begin{proposition}\label{T}
Let $T:\mathbb R^n \longrightarrow \mathbb R^n$ be a linear
bijective map. Then for every $\Psi\in \mathcal S'(\mathbb R^n)$ we
have $\supp(\Psi\circ T)=T^{-1}(\supp \Psi)$.
\end{proposition}

\begin{proof} As $T$ is a diffeomorphism, by definition of the composition with a distribution,
we have for $\phi\in \mathcal S(\mathbb R^n)$:
$$
\langle \Psi\circ T,\phi \rangle=|{\rm det} T^{-1}|\langle \Psi,
\phi\circ T^{-1} \rangle
$$

As $\Psi\in \mathcal S'(\mathbb R^n)\subset \mathcal D'(\mathbb
R^n)$ we have that $x\in\supp \Psi$ implies that for every
neighborhood $U$ of the origin there exists a function $\phi\in
C_c^\infty(x+U)$ such that $\langle \Psi,\phi \rangle\ne 0$.

We have then
$$
0\ne\langle \Psi,\phi \rangle = |{\rm det}T| \langle \Psi\circ T,
\phi\circ T \rangle
$$
where $\phi\circ T\in C_c^\infty(T^{-1}x+T^{-1}U)$, which implies
$T^{-1}x\in \supp(\Psi\circ T)$. This proves the inclusion
$T^{-1}(\supp \Psi)\subseteq \supp \Psi\circ T$. Clearly
$\Psi=(\Psi\circ T)\circ T^{-1}$, so the same argument with $T^{-1}$
instead of $T$ proves the opposite inclusion.
\end{proof}

\begin{remark}{\rm
$\Psi\circ T$ is the {\it pull-back} $T^*\Psi$ of $\Psi$, here for
analogy with the case where $\Psi$ is a function we shall write for
short $T(\Psi)$.
\par\medskip\noindent
In particular for $n=2d$ and $\mu,\nu\in \mathcal S'(\mathbb R^d)$,
we have that $x_1\in\supp \mu$, $x_2\in\supp \nu$ implies
$(\frac{x_1+x_2}{2},x_1-x_2)\in\supp(\mathcal
T_s(\mu\otimes\overline\nu))$, where $ \mathcal T_s$ is the
symmetric coordinate change operator defined in \eqref{Ts}.}
\end{remark}

More generally let $T:(x,y)\in\mathbb R^{2d}\longrightarrow
(u,v)=T(x,y)\in\mathbb R^{2d}$ be an invertible linear
transformation, with abuse of notation we shall still indicate with
$T$ the matrix associated with the transformation and the {\sl
change of coordinate operator} given by
$$
T: F(x,y)\longrightarrow (TF)(x,y)=F\circ T(x,y).
$$
with natural extension to distributions. We have then the following:

\begin{proposition}\label{4cases}
Suppose that $\mu, \nu \in \mathcal S'(\mathbb R^{2d})$, and
$T:\mathbb R^{2d}\longrightarrow\mathbb R^{2d}$ is a linear
invertible transformation. Let us write the inverse matrix of $T$
as:
$$
T^{-1}=\left(
         \begin{array}{cc}
           A & B \\
           C & D \\
         \end{array}
       \right)
$$
with $A, B, C, D$ submatrices of dimension $d\times d$. Then the
following hold:
$$
\begin{array}{cc}
(i)&   \Pi_1\supp T(\mu\otimes\nu) \hbox{\rm \ is an u.d. set}, {\rm
det} A\ne 0 \Longrightarrow \supp \mu \hbox{\rm \ is a u.d. set\ }\\
(ii)&  \Pi_1\supp T(\mu\otimes\nu) \hbox{\rm \ is an u.d. set}, {\rm
det} B\ne 0 \Longrightarrow \supp \nu \hbox{\rm \ is a u.d. set\ }\\
(iii)& \Pi_2\supp T(\mu\otimes\nu) \hbox{\rm \ is an u.d. set}, {\rm
det} C\ne 0 \Longrightarrow \supp \mu \hbox{\rm \ is a u.d. set\ }\\
(iv)&  \Pi_2\supp T(\mu\otimes\nu) \hbox{\rm \ is an u.d. set}, {\rm
det} D\ne 0 \Longrightarrow \supp \nu \hbox{\rm \ is a u.d. set\ }\\
\end{array}
$$
\end{proposition}

\begin{proof}
We prove (i), the others are analogous. By contradiction suppose
that $\supp\mu$ is not u.d.. Then for every $\epsilon>0$ there exist
$x,y\in \supp \mu$ such that $0<\|x-y\|<\epsilon$.

\noindent Let $z\in\supp \nu$, then $(x,z)$ and $(y,z)$ belong to
$\supp(\mu\otimes \nu)$.

\noindent Let $P=T^{-1}(x,z)\in \mathbb R^{2d}$ and
$Q=T^{-1}(y,z)\in \mathbb R^{2d}$ i.e.
$$
P=\left(
    \begin{array}{c}
      Ax+Bz \\
      Cx+Dz \\
    \end{array}
  \right); \hskip1cm
Q=\left(
  \begin{array}{c}
    Ay+Bz \\
    Cy+Dz \\
  \end{array}
\right),
$$
then, by Proposition \ref{T}, $P$ and $Q$ belong to $\supp
T(\mu\otimes\nu)$.

As ${\rm det}A\ne 0$ and $x-y\ne 0$, we have
$$
\begin{array}{rl}
0<\|A(x-y)\|    &     = \| \Pi_1 P - \Pi_1 Q\| \le \|P-Q\|   \\
                &     \le \|T^{-1}\| \|(x,z)-(y,z)\|=\|T^{-1}\| \|x-y\| < \|T^{-1}\| \ \epsilon.  \\
\end{array}
$$
Then $\Pi_1P$ and $\Pi_1Q$ are distinct points of $\Pi_1\supp
T(\mu\otimes\nu)$ with arbitrary small distance, i.e. $\Pi_1\supp
T(\mu\otimes\nu)$ is not a u.d. set.
\end{proof}

%The setting we have defined above is general enough be applied to
%most of the classical time-frequency representations.

As mentioned above, following \cite{Bayer}, we introduce next a {\it
Matrix-Wigner transform} which is a natural generalization of the
Wigner transform $W(\mu,\nu)=\mathcal F_2(\mathcal T_s(\mu\otimes
\nu))$ where the change of coordinates $\mathcal T_s$ have been
replaced by a general  bijective linear map $T$. This sesquilinear transform
turns out to be a quite comprehensive tool including most of the
basic time-frequency representations, we refer to \cite{Bayer} for
details and properties.
%See also \cite{tesi Carypis?} for a version in two dimensions in connection with the problems of ghost frequencies.

\begin{definition}
Let $T$ and $A, B, C, D$ be as before, then the {\sl Matrix-Wigner
transform} of $\mu, \nu \in \mathcal S'(\mathbb R^d)$ is defined as:
$$
W_T(\mu,\nu)=\mathcal F_2(T(\mu\otimes\overline\nu)).
$$
 As usual we shall write $W_T(\mu)$ for $W_T(\mu,\mu)$.
\end{definition}

In connection with our previous discussion we have the following
property.
\begin{proposition}\label{prop11}
If $\Pi_1\supp W_T(\mu,\nu)$ is a u.d. set, then
$$
\begin{array}{c}
{\rm det}A\ne0 \Longrightarrow \supp\mu \hbox{\rm \ is u.d},\\
{\rm det}B\ne0 \Longrightarrow \supp\nu \hbox{\rm \ is u.d}.\\
\end{array}
$$
In particular for the classical Wigner transform we have
$$
T^{-1}=\left(
         \begin{array}{cc}
           \frac{1}{2} {\rm Id} & \frac{1}{2} {\rm Id} \\
           {\rm Id} & {\rm -Id} \\
         \end{array}
       \right)
$$
where $\rm Id$ is the identity, therefore, as all subdeterminants
are non zero, we have that $\Pi_1\supp W(\mu,\nu)$ u.d. implies that
both $\supp\mu$ and $\supp\nu$ are u.d.
\end{proposition}

\begin{proof}
From Lemma \ref{Proj} we have $\Pi_1\supp
W_T(\mu,\nu)=\Pi_1\supp(T(\mu\otimes\overline\nu))$. Then
$\Pi_1\supp(T(\mu\otimes\nu))$ is u.d. and an application of
Proposition \ref{4cases} in the cases (i) or (ii) yields the thesis.
\end{proof}

Our next aim is to obtain an analogous of Proposition \ref{prop11}
for the case of the Fourier transforms of $\mu$ and $\nu$. Before
giving the statement we need the following remark on block matrices.

\begin{remark}\label{RemBlockMat}{\rm
Let $Z$ be an invertible $2d\times 2d$ matrix, and write
$$
Z=\left(
\begin{array}{cc}
Y & U \\
V &W
\end{array}
\right), \qquad Z^{-1}=\left(
\begin{array}{cc}
E & F \\
G &H
\end{array}
\right),
$$
where $Y,U,V,W,E,F,G,H$ are $d\times d$ matrices. From \cite[Theorem
2.1]{Lu_Shiou} we have that
\begin{equation}\label{BlockMat1}
\begin{split}
{\rm det}Y\ne0 &\Longrightarrow {\rm det}H\ne0,\\
{\rm det}W\ne0 &\Longrightarrow {\rm det}E\ne0.
\end{split}
\end{equation}
Since
$$
Z_1=\left(
\begin{array}{cc}
U & Y \\
W &V
\end{array}
\right) \quad\Longrightarrow\quad Z_1^{-1}=\left(
\begin{array}{cc}
G &H \\
E &F
\end{array}
\right),
$$
applying \eqref{BlockMat1} to $Z_1$ we also have
\begin{equation}\label{BlockMat2}
\begin{split}
{\rm det}U\ne0 &\Longrightarrow {\rm det}F\ne0,\\
{\rm det}V\ne0 &\Longrightarrow {\rm det}G\ne0.
\end{split}
\end{equation}
}
\end{remark}

\begin{proposition}\label{prop11FT}
If $\Pi_2\supp W_T(\mu,\nu)$ is a u.d. set and $T$ is an invertible
matrix satisfying
\begin{equation}\label{T-1}
T^{-1}=\left(
\begin{array}{cc}
A &B \\
C &D
\end{array}
\right)
\end{equation}
as in Proposition \ref{prop11}, then
$$
\begin{array}{c}
{\rm det}A\ne0 \Longrightarrow \supp\widehat{\nu} \hbox{\rm \ is u.d},\\
{\rm det}B\ne0 \Longrightarrow \supp\widehat{\mu} \hbox{\rm \ is u.d}.\\
\end{array}
$$
\end{proposition}

\begin{proof}
From \cite[Proposition 4]{Bayer} we have that
\begin{equation}\label{MWFourier1}
W_R(\widehat{\mu},\widehat{\nu})(x,\omega)=|{\rm \det}R|^{-1}
W_T(\mu,\nu)(-\omega,x)
\end{equation}
where
\begin{equation}\label{MWFourier2}
T=\left(
\begin{array}{cc}
{\rm Id} & 0 \\
0 &-{\rm Id}
\end{array}
\right) (R^{-1})^{\rm t} \left(
\begin{array}{cc}
0 &{\rm Id} \\
{\rm Id} &0
\end{array}
\right).
\end{equation}
Here and in the following, for convenience we write explicitly the
names of the variables in the Matrix-Wigner transform (of course,
when $\mu$ and $\nu$ are distributions, the notation $(-\omega,x)$
in \eqref{MWFourier1} stands for the corresponding linear change of
variables). A simple calculation shows that, from \eqref{MWFourier1}
and \eqref{MWFourier2} we have
\begin{equation}\label{MWFourier3}
W_T(\mu,\nu)(x,\omega)=|{\rm \det}T|^{-1}
W_R(\widehat{\mu},\widehat{\nu})(\omega,-x)
\end{equation}
where
\begin{equation}\label{MWFourier4}
R=\left(
\begin{array}{cc}
{\rm Id} & 0 \\
0 &-{\rm Id}
\end{array}
\right) (T^{-1})^{\rm t} \left(
\begin{array}{cc}
0 &{\rm Id} \\
{\rm Id} &0
\end{array}
\right).
\end{equation}
Now, from \eqref{T-1} and \eqref{MWFourier4} we get
$$
R=\left(
\begin{array}{cc}
C &A \\
-D &-B
\end{array}
\right),
$$
and $R$ is invertible since $T$ is invertible. Then, writing
\begin{equation}\label{R-1}
R^{-1}=\left(
\begin{array}{cc}
E &F \\
G &H
\end{array}
\right),
\end{equation}
from \eqref{BlockMat1} and \eqref{BlockMat2} we obtain
\begin{equation}\label{BlockMat3}
\begin{split}
&{\rm det}A\ne0 \Longrightarrow {\rm det}F\ne0,\qquad
{\rm det}B\ne0 \Longrightarrow {\rm det}E\ne0,\\
&{\rm det}C\ne0 \Longrightarrow {\rm det}H\ne0,\qquad {\rm det}D\ne0
\Longrightarrow {\rm det}G\ne0.
\end{split}
\end{equation}
Now, from \eqref{MWFourier3} we have
$$
\Pi_2\supp W_T(\mu,\nu) = \Pi_1\supp
W_R(\widehat{\mu},\widehat{\nu}),
$$
and so, by \eqref{R-1} and \eqref{BlockMat3} the thesis follows by
an application of Proposition \ref{prop11} to $\Pi_1\supp
W_R(\widehat{\mu},\widehat{\nu})$.
\end{proof}

From Propositions \ref{prop11} and \ref{prop11FT} we immediately
have the following corollary.
\begin{corollary}\label{cor:11+13}
Let $T$ be an invertible matrix satisfying
$$
T^{-1}=\left(
\begin{array}{cc}
A &B \\
C &D
\end{array}
\right).
$$
\begin{itemize}
\item[(i)]
Suppose that ${\rm det}A\ne0$ and ${\rm det}B\ne0$. If both
$\Pi_1\supp W_T(\mu,\nu)$ and $\Pi_2\supp W_T(\mu,\nu)$ are u.d.
sets, then $\supp\mu$, $\supp\nu$, $\supp\widehat{\mu}$ and
$\supp\widehat{\nu}$ are u.d. sets.
\item[(ii)]
Suppose that ${\rm det}A\ne0$ or ${\rm det}B\ne0$. If both
$\Pi_1\supp W_T(\mu)$ and $\Pi_2\supp W_T(\mu)$ are u.d. sets, then
$\supp\mu$ and $\supp\widehat{\mu}$ are u.d. sets.
\end{itemize}
\end{corollary}

%-------------------------------------------------------------------------
The basic connection between our setting and the hypothesis assumed
in \cite{olevskii} is the following remark.

\begin{remark}{\rm
In the case
$$
T=\left(
    \begin{array}{cc}
    {\rm Id} & 0 \\
      0 & {\rm -Id} \\
    \end{array}
  \right)
$$
the matrix-Wigner transform  is given by $W_T(\mu,\nu)=\mathcal
F_2(\mu\otimes\overline\nu)(x,\omega)=\mu(x)\overline{\widehat\nu(\omega)}$,
therefore Lev-Olevskii hypothesis (see \cite{olevskii})
$$
\mu=\sum_{\alpha\in\Lambda}a_\alpha \delta_\alpha; \ \ \
\widehat\nu=\sum_{\beta\in S}b_\beta\delta_\beta
$$ with $\Lambda, S$ u.d. sets, (and with $\mu=\nu$ in \cite{olevskii})
is a particular case of the hypothesis
$$
W_T(\mu,\nu)=\sum_{(r,s)\in A\times
B}c_{r,s}\delta_{(r,s)}
$$
where $A, B$ are u.d. sets.}
\end{remark}

The previous results, Propositions \ref{prop11}, \ref{prop11FT} and
Corollary \ref{cor:11+13}, where we can obtain u.d.ness of the
supports of signals from that of the their matrix-Wigner transform,
include many classical time-frequency transforms such as the STFT
and the Rihaczek transforms. As an example we consider the {\it
Ambiguity function}.

\begin{example}\label{ambiguity}{\rm
The Ambiguity function is defined as
$$
A(\mu,\nu)(x,\omega)=\int_{\mathbb R^d}e^{-2\pi i \omega
t}\mu(t+x/2)\overline{\nu(t-x/2)}\, dt
$$
for $\mu, \nu\in \mathcal S(\mathbb R^d)$, which is generalized to
$$
A(\mu,\nu)=W_T(\mu,\nu)$$ for $T=\left(
\begin{array}{cc}
    \frac{1}{2}{\rm Id} & {\rm Id} \\
    -\frac{1}{2}{\rm Id} & {\rm Id} \\
\end{array}
\right).$ Then
$\Pi_1\supp A(\mu,\nu)$ u.d. implies $\supp \mu$ and $\supp\nu$ u.d. In fact, it suffices to apply Proposition \ref{prop11}.}
\end{example}

Our next aim is to give a version of Theorem \ref{th:main} for the case of $W_T(\mu,\nu)$ for a $2\times
2$ matrix $T$, giving then a general picture of the situation for
the case $\mu,\nu\in\mathcal{S}'(\mathbb{R})$.

\par\medskip

Let $T:{\mathbb R}^2\to {\mathbb R}^2$ be an invertible linear transformation with inverse
$$T^{-1} = \left(\begin{array}{cc} a & b \\ c & d\end{array}\right).$$

%------------(Lemma 15) ----------------------------------------------------------------------------------------------------------------
%\begin{lemma}\label{lem:double-series2}
%Let $A, B$ u.d. subsets of ${\mathbb R}.$ Then for every $\alpha > 0$ there exists $\beta > 0$ such that
%$$
%\sum_{(r,s)\in T(A,B)}\left(1 + |r|\right)^\alpha\left(1 + |s|\right)^\alpha \left(1 + |ar+bs|\right)^{-\beta} \left(1 + |cr+ds|\right)^{-\beta} < +\infty.$$
%\end{lemma}
%\begin{proof}
%We put $T = \left(T_1, T_2\right).$ Then the series is given by
%$$
%\begin{array}{*2{>{\displaystyle}l}}
% \sum_{(u,v)\in A\times B} \left(1 + |T_1(u,v)|\right)^\alpha \left(1 + |T_2(u,v)|\right)^\alpha \left(1 + |u|\right)^{-\beta} \left(1 + |v|\right)^{-\beta}
%\\ \lesssim \sum_{(u,v)\in A\times B} \left(1 + |u| + |v|\right)^{2\alpha}\left(1 + |u|\right)^{-\beta} \left(1 + |v|\right)^{-\beta} \\ \leq \sum_{u\in A}\left(1 + |u|\right)^{\alpha - \beta}\cdot  \sum_{v\in B}\left(1 + |v|\right)^{\alpha - \beta},
%\end{array}
%$$ and the conclusion follows.
%\end{proof}
%------------------------------------------------------------------------------------------------------------------------------

\par\medskip
We assume from now on that $ab\neq 0.$
\begin{theorem}\label{measures}
Let $\mu, \nu\in {\mathcal S}^\prime({\mathbb
R})\setminus\left\{0\right\}$ satisfy $ W_T(\mu,\nu) = \sum_{(r,
s)\in A\times B} c_{r,s} \delta_{(r,s)}$ where $A, B$ are uniformly
discrete sets. Then $\mu$ and $\nu$ are measures supported in u.d. sets.
\end{theorem}
\begin{proof}
We apply Proposition \ref{prop11} to conclude that $\mu$ and
$\nu$ have u.d. support, denoted $S_\mu$ and $S_\nu$
respectively. According to Lemma \ref{lem:sum_deltas} we can put
$$\mu=\sum_{r\in S_\mu}\sum_{j=0}^N a_r^j \delta^{(j)}_r,\ \ \nu=\sum_{s\in S_\nu}\sum_{k=0}^M b_s^k \delta^{(k)}_s$$ with $a_r^j, b_s^k\in {\mathbb C}.$ We now assume $M\geq 1$ and $b_{s_0}^M\neq 0$ for some $s_0\in S_\nu$ and show that $\mu = 0,$ which is a contradiction from where we conclude that $\nu$ is a measure.
\par\medskip
For any real-valued functions $\phi_1, \phi_2 \in {\mathcal S}({\mathbb R})$ we have, for $\phi = \phi_1\otimes\phi_2,$
$$
\langle W_T(\mu, \nu), \phi\rangle = \langle T\left(\mu\otimes
\overline{\nu}\right), {\mathcal F}_2^{-1}\phi\rangle = |\mbox{det}\
T^{-1}| \langle \mu\otimes \overline{\nu}, \left({\mathcal
F}_2^{-1}\phi\right)\circ T^{-1}\rangle.$$ Hence
$$
|\mbox{det} T| \langle W_T(\mu, \nu), \phi\rangle = \langle \mu_u,
\langle \nu_v, \left( \phi_1 \otimes
 \overline{\widehat{\phi_2}} \right)\circ
T^{-1}(u,v)\rangle\rangle.$$ Since
$$
\begin{array}{*1{>{\displaystyle}c}}
\langle \nu_v, \left(\phi_1 \otimes
\overline{\widehat{\phi_2}}
\right)\circ T^{-1}(u,v)\rangle = \\ \\
\sum_{s\in S_\nu}\sum_{k=0}^M (-1)^k b_s^k \sum_{m=0}^k{k\choose m}
\phi_1^{(m)}(au+bs)b^m \widehat{\phi_2}^{(k-m)}(cu+ds)d^{k-m}\end{array}$$ we finally
obtain

\begin{equation}\label{eq:wigner2}
\begin{array}{*1{>{\displaystyle}c}}
|\mbox{det} T| \langle W_T(\mu, \nu), \phi\rangle = \\ \\
\sum_{j=0}^N\sum_{\ell=0}^j\sum_{k=0}^M\sum_{m=0}^k
\lambda^{\ell,m}_{j,k}\sum_{r\in S_\mu}\sum_{s\in S_\nu}a_r^j
\overline{b_s^k}\phi_1^{(m+\ell)}(ar+bs)\overline{\widehat{\phi_2}}^{(k+j-m-\ell)}(cr+ds),
\end{array}
\end{equation} where
$$
\lambda^{\ell,m}_{j,k} = (-1)^{k+j}{k\choose m}{j\choose \ell}a^\ell
b^m c^{j-\ell} d^{k-m}.$$

%---------- (convergenza assoluta serie doppia) ----------------------------------------------------------------------------------
%
%We already know that {\color{red}controllare}
%$$
%(r,s) = T\left(ar+bs, cr+ds\right)\in T(A\times B)\ \ \forall (r,s)\in S_\mu \times S_\nu.$$ Since
%\begin{equation}\label{eq:bound-coeff2}
%    \sup_{r\in S_\mu}\left|a^j_r\right|\left(1+|r|\right)^{-N-1} < \infty,\ \ \sup_{s\in S_\nu}\left|b^k_s\right|\left(1+|s|\right)^{-M-1} < \infty
%\end{equation} for every $0\leq j\leq N$ and $0\leq k\leq M,$ an application of Lemma \ref{lem:double-series2} permits to conclude that the double series
%in the right hand side of (\ref{eq:wigner2}) is absolutely
%convergent.
%
%--------------------------------------------------------------------------------------------

We show next that for every fixed $j, l, k, m$ the double sum over
$r,s$ in \eqref{eq:wigner2} is absolutely convergent. We recall the
following facts.
\\
There exists $\alpha>0$ such that, setting for short $\langle t
\rangle=1+|t|,$
\begin{equation}\label{eq:bound<r>} \sup_{r\in
S_\mu}\left|a^j_r\right|\langle r \rangle^{-\alpha} < \infty,\ \
\sup_{s\in S_\nu}\left|b^k_s\right|\langle s \rangle^{-\alpha} <
\infty.
\end{equation}
The functions $\phi_1^{(m+\ell)}$and
$\widehat{\phi}_2^{(k+j-m-\ell)}$ are in $\mathcal
S(\mathbb R)$, therefore, for every $\beta>0$ and suitable constant
$C>0$, we have
$$|\phi_1^{(m+\ell)}(ar+bs)|\le C \langle ar+bs
\rangle^{-\beta}, \ \
|\widehat{\phi}_2^{(k+j-m-\ell)}(cr+ds)|\le C\langle
cr+ds\rangle^{-\beta}.$$ Furthermore we remark that linear
bijections on $\mathbb R^2$ are bi-continuous and therefore preserve
u.d. sets, so that $T^{-1}(S_\mu\times S_\nu)$ is u.d. Supposing then
that $T$ has matrix
$$T = \left(\begin{array}{cc} a' & b' \\ c' & d'\end{array}\right),$$
and indicating by $C$ generic (possibly different) suitable
constants, we have:
\begin{equation*}
\begin{array}{*1{>{\displaystyle}l}}
\ \ \ \ \sum_{r\in S_\mu}\sum_{s\in S_\nu} |a_r^j| \, |b_s^k| \,
|\phi_1^{(m+\ell)}(ar+bs)| \,
|\widehat{\phi}_2^{(k+j-m-\ell)}(cr+ds)|
\\
\le C \sum_{r\in S_\mu}\sum_{s\in S_\nu}  \langle r
\rangle^{\alpha}\, \langle s \rangle^{\alpha} \, \langle ar+bs
\rangle^{-\beta} \, \langle cr+ds\rangle^{-\beta}
\\
\le C \sum_{(u,v)\in T^{-1}(S_\mu\times S_\nu)}  \langle a'u+b'v
\rangle^{\alpha}\, \langle c'u+d'v \rangle^{\alpha} \, \langle u
\rangle^{-\beta} \, \langle v \rangle^{-\beta}
\\
\le C \sum_{(u,v)\in T^{-1}(S_\mu\times S_\nu)}  \langle u
\rangle^{\alpha}\, \langle v \rangle^{\alpha}\, \langle u
\rangle^{\alpha} \, \langle v \rangle^{\alpha} \, \langle u
\rangle^{-\beta} \, \langle v \rangle^{-\beta}
%\\
%= C \sum_{(u,v)\in T^{-1}(S_\mu\times S_\nu)}  \langle u
%\rangle^{2\alpha-\beta}\, \langle v \rangle^{2\alpha-\beta}
\\
\le C \sum_{(u,v)\in T^{-1}(S_\mu\times S_\nu)}
(1+|u|^2+|v|^2)^{(2\alpha-\beta)},
\end{array}
\end{equation*}
where $\beta$ can be chosen arbitrary large and $T^{-1}(S_\mu\times
S_\nu)$ is a u.d. set, and therefore the last sum is convergent,
i.e. \eqref{eq:wigner2} is absolutely convergent.

%--------------------------------------------------------------------------------------------

Let us take now $$0 < \varepsilon < \delta(A):= \inf\{|r -
r^\prime|:\ r, r^\prime\in A,\ r\neq r^\prime\}.$$ Fix $r_0 \in
S_\mu, s_0\in S_\nu$, put $x_0 = ar_0+bs_0$ and choose
$\phi_1\in {\mathcal D}(x_0-\varepsilon, x_0+\varepsilon)$ such that
$\phi_1^{(n)}(x_0)=0$ for $n=0, \dots , N+M-1$ whereas
$\phi^{(N+M)}_1(x_0)\neq 0.$ Then, for $\phi_2 \in {\mathcal
S}({\mathbb R})$ (real valued), we have
$$
|\mbox{det} T| \langle W_T(\mu, \nu), \phi\rangle = a^N
b^M\phi_1^{(N+M)}(x_0)\sum_{(r,s) \in D}a^N_r
\overline{b_s^M}\overline{\widehat{\phi}_2}(cr+ds).$$
Here $D = \{(r,s)\in S_\mu\times S_\nu:\ ar+bs = x_0\}.$ Hence, if
$\phi_2$ has also compact support,
$$
\begin{array}{*2{>{\displaystyle}l}}
    \left| \phi_1^{(N+M)}(x_0)\sum_{(r, s) \in D}a^N_r
\overline{b_s^M}\widehat{\phi}_2(cr+ds)\right|
& = |a|^{-N}|b|^{-M}|\mbox{det} T|\left| \langle
W_T(\mu,\nu), \phi_1\otimes \phi_2
    \rangle \right|\\ & \\ & \leq C ||\phi_1||_\infty
    ||\phi_2||_\infty,\end{array}$$ where the constant $C$ depends on the (compact) support of $\phi_1\otimes \phi_2.$ Next, we fix  $\psi \in {\mathcal D}(-\varepsilon, \varepsilon)$
such that $\psi^{(n)}(0)=0$ for $n=0, \dots , N+M-1$ and
$\psi^{(N+M)}(0)=1,$ and for each $t\geq 1,$ let us consider
$\psi^t(x):=\psi(t x)$ and $\phi^t_1(x)=\psi^t(x-r_0).$ As in the proof of Theorem \ref{th:main} we conclude that
\begin{equation}\label{eq:sum-02}
\sum_{(r,s)\in
D}a_r^N\overline{b_s^M}\widehat{\phi_2}(cr+ds)
= 0
\end{equation} for every real valued $\phi_2\in {\mathcal D}({\mathbb R}).$ To discuss the meaning of the obtained expression we need some notation. Denote
$$
S = \left\{r\in S_\mu:\ (r,s)\in D\ \mbox{for some}\ s\in S_\nu\right\}.$$ For $r\in S$ we denote by $s(r)$ the unique $s\in S_\nu$ such that $(r,s)\in D.$ Observe that there are constants $\alpha, \beta$ such that $\alpha\neq 0$ and $cr+ds = \alpha r + \beta$ whenever $(r,s)\in D.$ From (\ref{eq:bound<r>}) and the fact that $S$ is a u.d. set it follows that
$$
\sum_{r\in S}\overline{a_r^N}b^M_{s(r)}\delta_{\alpha r + \beta}$$ defines a tempered distribution. Condition (\ref{eq:sum-02}) means that the Fourier transform of that distribution vanishes. Consequently
$$
 \overline{a_r^N}b^M_{s(r)} = 0\ \ \forall r\in S.$$ Since $s(r_0) = s_0$ and $b_{s_0}^M\neq 0$ we conclude $a_{r_0}^N = 0.$ Since $r_0\in S_\mu$ is arbitrary we conclude $a_r^N = 0$ for every $r\in S_\mu.$ Proceeding by recurrence on the order of $\mu$ we finally get that $a_r^j = 0$ for all $r\in S_\mu$ and $0\leq j\leq N,$ from where it follows $\mu = 0.$ This contradiction proves that $M = 0$ and $\nu$ is a measure. The same argument but swapping the role of the distributions proves that also $\mu$ is a measure.
\end{proof}

\begin{corollary}\label{cor:Theorem_Fourier_side}
Under the same hypothesis as in Theorem \ref{measures}
%Let $\mu, \nu\in {\mathcal S}^\prime({\mathbb
%R})\setminus\left\{0\right\}$ satisfy $ W_T(\mu,\nu) = \sum_{(r,
%s)\in A\times B} c_{r,s} \delta_{(r,s)}$ where $A, B$ are uniformly
%discrete sets. Then
$\widehat{\mu}$ and $\widehat{\nu}$ are measures supported on
u.d. sets.
\end{corollary}

\begin{proof} By \cite[Proposition 4, (ii)]{Bayer}
$$W_S(\widehat{\mu},\widehat{\nu})(x,\omega)=|{\rm det}S|^{-1}W_T(\mu,\nu)(-\omega,x)$$ where $$S=\left(\begin{array}{cc}1 & 0 \\0 & -1\end{array}\right)\left(\begin{array}{cc} a & b \\ c & d\end{array}\right)^t\left(\begin{array}{cc}0 & 1 \\1 & 0\end{array}\right).$$
Hence, $$S^{-1}=\left(\begin{array}{cc} \tilde{a} & \tilde{b} \\
\tilde{c} & \tilde{d}\end{array}\right)$$ where, as
$ab\neq 0$, by direct calculation or from \eqref{BlockMat1},
\eqref{BlockMat2}, we have $\tilde{a}\tilde{b}\neq 0$, and we can
then apply Theorem \ref{measures} to conclude.
\end{proof}

\begin{corollary}\label{cor:structure}
Let $\mu, \nu\in {\mathcal S}^\prime({\mathbb R})\setminus\left\{0\right\}$ satisfy
$ W_T(\mu,\nu) = \sum_{(r, s)\in A\times B} c_{r,s} \delta_{(r,s)}$ where $A, B$ are u.d. sets. Then, there are $a, \, b > 0$ such that $\mu$ is a finite linear combination of time-frequency shifts of $\sum_{n\in {\mathbb Z}}\delta_{na}$ and  $\nu$ is a finite linear combination of time-frequency shifts of $\sum_{n\in {\mathbb Z}}\delta_{nb}.$

\end{corollary}

\begin{proof} Theorem \ref{measures} and Corollary \ref{cor:Theorem_Fourier_side} imply that $\mu, \, \nu, \, \widehat{\mu}, \, \widehat{\nu}$ are measures supported on u.d. sets in ${\mathbb R}.$ We now apply \cite[Theorems 1,3]{olevskii} to conclude.

\end{proof}

\end{document}